\documentclass[12pt,compress]{article}
\usepackage{mathrsfs}
\usepackage{amsfonts}

\usepackage{amsmath}
\usepackage{amssymb}
\usepackage{amsthm}
\usepackage{enumerate}

\usepackage{longtable,booktabs,makecell,multirow,tabularx}
\usepackage{graphicx}

\usepackage{geometry}
\newtheorem{theorem}{Theorem}[section]
\newtheorem{lemma}[theorem]{Lemma}

\numberwithin{equation}{section}

\def\UU{{\mathfrak U}}
\def\FF{{\mathfrak F}}

\def\LL{{\mathscr L}}

\usepackage{latexsym}
\usepackage{amssymb}

\usepackage{color}
\usepackage{hyperref}

\usepackage{cleveref}
\usepackage[numbers,sort&compress]{natbib}

\newtheorem{innercustomgeneric}{\customgenericname}
\providecommand{\customgenericname}{}
\newcommand{\newcustomtheorem}[2]{%
	\newenvironment{#1}[1]
	{%
		\renewcommand\customgenericname{#2}%
		\renewcommand\theinnercustomgeneric{##1}%
		\innercustomgeneric
	}
	{\endinnercustomgeneric}
}

\newcustomtheorem{customtheorem}{Theorem}
\newcustomtheorem{customlemma}{Lemma}
\newcustomtheorem{customdefinition}{Definition}
\newcustomtheorem{customhypothesis}{Hypothesis}
\newcustomtheorem{customcorollary}{Corollary}
\newcustomtheorem{customproposition}{Proposition}
\newcustomtheorem{customremark}{Remark}
\newcustomtheorem{customquestion}{Question}
\newcustomtheorem{customproblem}{Problem}
\newcustomtheorem{customexample}{Example}
\newcustomtheorem{customConjecture}{Conjecture}
\newcustomtheorem{customconclusion}{Conclusion}

\theoremstyle{definition}
\newtheorem{definition}[theorem]{Definition}

\newtheorem{example}[theorem]{Example}

\theoremstyle{definition}

\usepackage[numbers,sort&compress]{natbib}

\begin{document}
\begin{sloppypar}




\begin{center}
{\Large \bf{On the partial $ \mathscr L $-$ \Pi $-property of subgroups of finite groups}

\renewcommand{\thefootnote}{\fnsymbol{footnote}}

\footnotetext[1]
{Corresponding author.}

}\end{center}

                        \vskip0.6cm
\begin{center}

                       Zhengtian Qiu$ ^{1} $ and Adolfo Ballester-Bolinches$^{2, \ast}$

                            \vskip0.5cm

       $ ^{1} $School of Mathematics and Statistics, Southwest University,

       Chongqing 400715, People's Republic of China

       $ ^{2} $Departament de Matem$ \grave{\mathrm{a}} $tiques, Universitat de Val$ \grave{\mathrm{e}} $ncia, Dr. Moliner 50,
       46100 Burjassot, Valencia, Spain

      E-mail addresses:  qztqzt506@163.com \,    \  Adolfo.Ballester@uv.es

\end{center}

                          \vskip0.5cm

\begin{abstract}
	Let $ H $ be  a subgroup of a finite group $ G $.  We say that $ H $ satisfies the partial $ \mathscr L $-$ \Pi $-property in $ G $ if $ H\unlhd G $, or if $ | G / K : \mathrm{N} _{G / K} (HK/K)| $ is a $ \pi (HK/K) $-number for any  $ G $-chief factor  of type $ H^{G}/K $ with $ H_{G}\leq K $.  In this paper, we investigate the structure of  finite groups   under the assumption that some subgroups of prime power order satisfy  the partial $ \mathscr L $-$ \Pi  $-property.

\end{abstract}

{\hspace{2em} \small \textbf{Keywords:} Finite group, $ p $-soluble group, the partial $ \mathscr L $-$ \Pi $-property.}
	
\vskip0.1in

{\hspace{2em} \small \textbf{Mathematics Subject Classification (2020):} 20D10, 20D20.}

\section{Introduction}


All groups considered in this paper are finite. We use conventional notions as in \cite{MR1169099} or \cite{Gorenstein-1980}.

Throughout the paper,  $ G $ always denotes a finite group, $ p $ denotes a fixed prime, $ \mathrm{C}_{p} $ denotes a cyclic group of order $ p $, $ \pi $ denotes some set of primes  and $ \pi(G) $ denotes the set of all primes dividing $ |G| $. An integer $ n $ is called a $ \pi $-number if all prime divisors of $ n $ belong to $ \pi $. In particular, an integer $ n $ is called a $ p $-number if it is a power of $ p $. The symbol $ \mathbb{F}_{p} $ denotes the finite field with $ p $ elements. Recall that the $ p $-rank  of a $ p $-soluble group $ G $ is the largest integer $ k $ such that $ G $ has a chief factor of order $ p^{k} $. Note that a group $ G $ is $ p $-supersoluble if and only if $ G $ is $ p $-soluble with $ p $-rank at most $ 1 $.


Recall that a class $ \FF $  of groups is called a formation if $ \FF $ is closed under taking homomorphic images and subdirect products.  A formation $ \FF $ is said to be saturated if $ G/\Phi(G) \in \FF $  implies that $ G \in \FF $.  Throughout  this paper, we  use $ \UU $ (respectively, $ \UU_{p} $) to denote the class of supersoluble (respectively, $ p $-supersoluble) groups.

Let $ \FF $ be a formation.  A chief factor $ L/K $ of a group $ G $ is said to be $ \FF $-central in $ G $ if $ L/K \rtimes G/C_{G}(L/K) \in \FF $.   A normal subgroup $ N $ of $ G $ is called $ \FF $-hypercentral in $ G $ if either $ N = 1 $ or every $ G $-chief factor below $ N $ is $ \FF $-central in $ G $. Let $ \mathrm{Z}_{\FF}(G) $ denotes the $ \FF $-hypercentre of $ G $, that is, the product of all $ \FF $-hypercentral normal subgroups of $ G $.




It is interesting to investigate the influence of some kind of embedding properties of subgroups on the structure of a finite group and many significant results have been given (for example, see \cite{Adolfo-2011, Adolfo-JPA, Adolfo-2015, Qian-2021}). In 2013, Chen and Guo \cite{Chen-2013} introduced the concept of the \emph{partial $ \Pi $-property} of subgroups of finite groups,  which generalizes a large number of known embedding
properties (see \cite[Section 7]{Chen-2013}). Let $ H $ be a subgroup of a group $ G $. We say that $ H $ satisfies the \emph{partial $ \Pi $-property} in $ G $ if  there exists a chief series $ \varGamma_{G}: 1 =G_{0} < G_{1} < \cdot\cdot\cdot < G_{n}= G $ of $ G $ such that for every $ G $-chief factor $ G_{i}/G_{i-1} $ $ (1\leq i\leq n) $ of $ \varGamma_{G} $, $ | G / G_{i-1} : \mathrm{N} _{G/G_{i-1}} (HG_{i-1}/G_{i-1}\cap G_{i}/G_{i-1})| $ is a $ \pi (HG_{i-1}/G_{i-1}\cap G_{i}/G_{i-1}) $-number. They obtained the following results by assuming that some maximal or minimal subgroups of a Sylow subgroup  satisfiy the partial $ \Pi $-property.

\begin{theorem}[{\cite[Proposition 1.4]{Chen-2013}}]\label{maximal}
	Let $ E $ be a normal subgroup of a group $ G $ and let $ P $ be a Sylow $ p $-subgroup of $ E $. If every maximal subgroup of $ P $ satisfies the partial $ \Pi $-property in $ G $, then either $ E\leq \mathrm{Z}_{\UU_{p}}(G) $ or $ |E|_{p}=p $.
\end{theorem}

\begin{theorem}[{\cite[Proposition 1.6]{Chen-2013}}]\label{minimal}
	Let $ E $ be a normal subgroup of a group $ G $ and let $ P $ be a Sylow $ p $-subgroup of $ E $. If every cyclic subgroup of $ P $ of prime order or order $ 4 $ (when $ P $ is
	not quaternion-free)  satisfies the partial $ \Pi $-property in $ G $, then  $ E\leq \mathrm{Z}_{\UU_{p}}(G) $.
\end{theorem}


From the definition of the partial $ \Pi $-property of subgroups, all $ G $-chief factors of $ \varGamma_{G} $ have been  considered.  Now we introduce the following concept by considering fewer $ G $-chief factors, which generalizes the partial $ \Pi $-property of subgroups (see Lemma \ref{passing}).

\begin{definition}
	Let $ H $ be  a subgroup of a group $ G $, we say that $ H $ satisfies the \emph{partial $ \LL $-$ \Pi $-property}  in $ G $ if $ H\unlhd G $, or if $ | G / K : \mathrm{N} _{G / K} (HK/K)| $ is a $ \pi (HK/K) $-number for any  $ G $-chief factor  of type $ H^{G}/K $ with $ H_{G}\leq K $, which means that $K$ is a maximal $G$-invariant subgroup of $H^{G}$ with $ H_{G}\leq K $.
\end{definition}


According to Lemma \ref{passing}, we see that if a subgroup $ H $ of a group $ G $ satisfies the partial $ \Pi $-property in $ G $, then $ H $ satisfies the partial $ \LL $-$ \Pi $-property in $ G $.  But the converse fails in general. Let us see the following  example.

\begin{example}
	Let $ G = \mathrm{Sym}(4) $ be the symmetric group of degree $ 4 $. It is easy to see that every subgroup  of $ G $ of order $ 4 $ satisfies the partial $ \LL $-$ \Pi $-property in $ G $.  However, $ \langle (1234)\rangle $ does not satisfy the partial $ \Pi $-property in $ G $.
\end{example}

At first, we can get the following result by assuming that some minimal subgroups satisfy the partial $ \LL $-$ \Pi $-property.

\begin{customtheorem}{A}\label{super}
	Let $ E $ be a normal subgroup of a group $ G $ and let $ P $ be a Sylow $ p $-subgroup of $ E $. If every cyclic subgroup of $ P $ of prime order or order $ 4 $ (when $ P $ is not quaternion-free)  satisfies the partial $\LL$-$ \Pi $-property in $ G $, then  $ E\leq \mathrm{Z}_{\UU_{p}}(G) $.
\end{customtheorem}







Here, we study the structure of a group $ G $ with some subgroups of prime power order satisfying the partial $ \LL $-$ \Pi $-property.

  \begin{customtheorem}{B}\label{two}
  	Let $ P \in {\rm Syl}_{p}(G) $ and let $ d $ be a  power of $ p $ such that $ p \leq d \leq \sqrt{|P|} $. Assume that every subgroup of $ P $ of order $ d $ satisfies the partial $ \LL $-$ \Pi $-property in $ G $, and assume further that every cyclic subgroup of $ P $ of order $ 4 $  satisfies the partial $ \LL $-$ \Pi $-property in $ G $ when $ d = 2 $ and $ P $  is not quaternion-free. Then $ G $ is
  	$ p $-soluble and $ G/\mathrm{O}_{p',p}(G) $ is $ p $-supersoluble. Furthermore, if the $p$-rank of $G$ is greater than $1$, then
  	
  	\begin{enumerate}[\rm(1)]
  		\item $ d\geq p^{2}|P\cap \mathrm{O}_{p',\Phi}(G)| $, where $ \mathrm{O}_{p',\Phi}(G) $ is a normal subgroup of $ G $ such that $ \mathrm{O}_{p',\Phi}(G)/\mathrm{O}_{p'}(G)= \Phi(G/\mathrm{O}_{p'}(G)) $;
  		\item $\mathrm{O}_{p',p}(G)/\mathrm{O}_{p',\Phi}(G)$  is a homogeneous $ \mathbb{F}_{p}[G] $-module with all its irreducible $ \mathbb{F}_{p}[G] $-submodules being not absolutely irreducible and having
  		dimension $ k>1 $ such that $ k\big | \log_p \frac{d}{|P\cap \mathrm{O}_{p',\Phi}(G)|} $.
  	\end{enumerate}
  \end{customtheorem}

  In fact, we cannot obtain the $ p $-supersolubility of $ G $ in Theorem \ref{two}. Let us see the following example.

  \begin{example}
  	Consider an elementary abelian group $ U=\langle x, y|x^{5}=y^{5}=1, xy=yx \rangle $ of order $ 25 $.  Let $ \alpha $ be an automorphism of $ U $ of order $ 3 $ such that $ x^{\alpha}=y, y^{\alpha}=x^{-1}y^{-1} $. Let  $ V=\langle a, b \rangle $ be a copy of $ U $ and $ G = (U\times V)\rtimes \langle \alpha \rangle $. For any subgroup $ H $ of $ G $ of order $ 25 $, there exists a minimal normal subgroup $ K $ such that $ H\cap K = 1 $ (for details, see \cite[Example]{Guo-Xie-Li}). Then $ G $ has a  chief series $$ \varGamma_{G}: 1 =G_{0} <K= G_{1} < HK=G_{2} < G_{3}= G $$  such that for every $ G $-chief factor $ G_{i}/G_{i-1} $ $(1\leq i\leq 3) $ of $ \varGamma_{G} $, $ | G:\mathrm{N} _{G} (HG_{i-1}\cap G_{i})| $ is a $ 5 $-number. It means that $ H $ satisfies the partial $ \Pi $-property in $ G $, and of course, $ H $ also satisfies the partial $ \LL $-$ \Pi $-property in $ G $. However, $ G $ is not $ 5 $-supersoluble.
  \end{example}

  At last, we state the characterisation of a group $ G $ in which every subgroup  of order $ p^{2} $  satisfies the partial $ \LL $-$ \Pi $-property. Without loss
  of generality, we may assume that the Sylow $ p $-subgroups of $ G $ have order at least $ p^{2} $ and $ \mathrm{O}_{p'}(G) = 1 $. 

  \begin{customtheorem}{C}\label{three}
  	Let $ P $ be a Sylow $ p $-subgroup of a group $ G $. Assume that $ \mathrm{O}_{p'}(G)=1 $ and $ |P|\geq p^{2} $. If every subgroup of $ P $ of order $ p^{2} $ satisfies the partial $ \LL $-$ \Pi $-property in $ G $, then one of the following statements holds:
  	\begin{enumerate}[\rm(1)]
  		\item $ G $ is $ p $-supersoluble;
  		\item $ P $ is a minimal normal subgroup of $ G $ of order $ p^{2} $;
  		\item $ p=2 $ and $ G\cong \mathrm{Sym}(4) $;
  		\item $ p = 3 $, $  G = V\rtimes H $, where $ H\cong \mathrm{GL}(2,3) $ or $ \mathrm{SL}(2,3) $ acts naturally on $ V\cong \mathrm{C}_{3}\times \mathrm{C}_{3} $;
  		\item $ |P|\geq p^{4} $, $ P\unlhd G $, $ G/P $ is cyclic, and $ P $ is a homogeneous $ \mathbb{F}_{p}[G] $-module with all its irreducible $ \mathbb{F}_{p}[G] $-submodules having dimension $ 2 $.
  	\end{enumerate}
  \end{customtheorem}

\section{Preliminaries}

In this section, we give some lemmas that will be used in our proofs.

\begin{lemma}[{\cite[Lemma 2.1]{Chen-2013}}]\label{over}
	Let $ H $ be a subgroup of a group $ G $ and $ N\unlhd G $. If either $ N \leq H $ or $ \gcd(|H|, |N|)=1 $ and $ H $ satisfies the partial $ \Pi $-property
	in $ G $, then $ HN/N $ satisfies the partial $ \Pi $-property in $ G/N $.
\end{lemma}

\begin{lemma}\label{passing}
	Let $ H $ be a subgroup of a group $ G $. Then $ H $ satisfies the partial $ \Pi $-property in $ G $  if and only if there exists a chief series $ \varOmega_{G} $ of $ G $ passing through $ H_{G} $ and $ H^{G} $ such that  $ | G / H_{i-1} : \mathrm{N} _{G/H_{i-1}} (HH_{i-1}/H_{i-1}\cap H_{i}/H_{i-1})| $ is a $ \pi (HH_{i-1}/H_{i-1}\cap H_{i}/H_{i-1}) $-number for each $ G $-chief factor $ H_{i}/H_{i-1} $ of $ \varOmega_{G} $ with $ H_{G}\leq H_{i-1}\leq H_{i}\leq H^{G} $.
\end{lemma}

\begin{proof}[\bf{Proof}]
	We only need to prove the necessity. Suppose that $ H $ satisfies the partial $ \Pi $-property in $ G $. It is no loss to assume that $ H\ntrianglelefteq G $.  Then $ H_{G}<H<H^{G} $. Assume that $ H_{G} > 1 $. By a routine check, we know that $ (H/H_{G})^{G/H_{G}}=H^{G}/H_{G} $ and $ (H/H_{G})_{G/H_{G}}=H_{G}/H_{G} $. Write $ \overline{G}=G/H_{G} $. In view of Lemma \ref{over}, $ H/H_{G} $ satisfies the partial $ \Pi $-property in $ G/H_{G} $. By induction, there exists a $ G/H_{G} $-chief series $$ G/H_{G}>\cdots >H^{G}/H_{G}:=H_{k}/H_{G}>\cdots >H_{i}/H_{G}>H_{i-1}/H_{G}>\cdots >H_{G}/H_{G}:=H_{0}/H_{G} $$ passing through $ H^{G}/H_{G} $ and $ H_{G}/H_{G} $ such that $ |\overline{G}/\overline{H_{i-1}}:\mathrm{N}_{\overline{G}/\overline{H_{i-1}}}(\overline{H}\overline{H_{i-1}}/\overline{H_{i-1}}\cap \overline{H_{i}}/\overline{H_{i-1}})| $ is a $ \pi(\overline{H}\overline{H_{i-1}}/\overline{H_{i-1}}\cap \overline{H_{i}}/\overline{H_{i-1}}) $-number for each
	$ \overline{G}/\overline{H_{i-1}} $-chief factor $ \overline{H_{i}}/\overline{H_{i-1}} $  with $ i=0, ..., k $. Then we can get the following
	$ G $-chief series
	
	$$  G>\cdots >H^{G}=H_{k}>\cdots >H_{i}>H_{i-1}>\cdots>H_{G}=H_{0}>\cdots >1. $$
	Therefore, $ | G / H_{i-1} : \mathrm{N} _{G/H_{i-1}} (HH_{i-1}/H_{i-1}\cap H_{i}/H_{i-1})| $ is a $ \pi (HH_{i-1}/H_{i-1}\cap H_{i}/H_{i-1}) $-number for each $ G $-chief factor $ H_{i}/H_{i-1} $  with $ 1\leq i\leq k $, as desired.

Now assume that $ H_{G}=1 $. There exists a chief series $ \varGamma_{G}: 1 =G_{0} < G_{1} < \cdot\cdot\cdot < G_{n}= G $ of $ G $ such that for every $ G $-chief factor $ G_{i}/G_{i-1} $ $ (1\leq i\leq n) $ of $ \varGamma_{G} $, $ | G / G_{i-1} : \mathrm{N} _{G/G_{i-1}} (HG_{i-1}/G_{i-1}\cap G_{i}/G_{i-1})| $ is a $ \pi (HG_{i-1}/G_{i-1}\cap G_{i}/G_{i-1}) $-number.  Note that $$ \varGamma_{G}\cap H^{G}: 1=G_{0}\cap H^{G}<G_{1}\cap H^{G}<\cdots <G_{n}\cap H^{G}=H^{G}  $$ is, avoiding repetitions, part of a chief
	series of $ G $. Set $ G_{i} \cap H^{G} = H_{i} $, where $ i=0, 1, ..., n $. We can complete $ \varGamma_{G} \cap H^{G} $ to obtain a chief series $ \varOmega_{G} $ of $ G $. Then $ G $ has a chief series $$ \varOmega_{G}: 1=H_{0}<H_{1}<\cdots <H_{n}=H^{G}<\cdots <G $$ passing through $ H^{G} $.
	Observe that $ \mathrm{N}_{G}(HG_{i-1}\cap G_{i})\leq \mathrm{N}_{G}(HH_{i-1} \cap H_{i}) $ and $ (HG_{i-1}\cap G_{i})/G_{i-1}\cong (H\cap G_{i})/(H\cap G_{i-1})\cong (HH_{i-1}\cap H_{i})/H_{i-1} $. Hence $ | G / H_{i-1} : \mathrm{N} _{G/H_{i-1}} (HH_{i-1}/H_{i-1}\cap H_{i}/H_{i-1})| $ is a $ \pi (HH_{i-1}/H_{i-1}\cap H_{i}/H_{i-1}) $-number for any $i=1, ..., n$, as desired.
Our proof is now complete.
\end{proof}

\begin{lemma}\label{OV}
	Let $ H $ be a subgroup of a group $ G $ and $ N \unlhd G $. We have:
	
	\begin{enumerate}[\rm(1)]
		\item If $ H $ satisfies the partial $ \LL $-$ \Pi $-property in $ G $, then $ HN $ satisfies the partial $ \LL $-$ \Pi $-property in $ G $;
		\item If $ H $ satisfies the partial $ \LL $-$ \Pi $-property in $ G $, then $ HN/N $ satisfies the partial $ \LL $-$ \Pi $-property in $ G/N $.
	\end{enumerate}

\end{lemma}

\begin{proof}[\bf{Proof}]
	
	
	(1) If $ HN\unlhd G $, there is nothing to prove. Hence we may assume that $ HN\ntrianglelefteq G $. Therefore $ H\ntrianglelefteq G $.  Clearly, $ (HN)^{G}=H^{G}N $ and $  H_{G}N\leq  (HN)_{G} $. Let $ M $ be a maximal $ G $-invariant subgroup  of $ (HN)^{G} $ with $ (HN)_{G}\leq M $.  Note that $ H_{G}\leq H^{G}\cap M $ and $ (HN)^{G}/M=H^{G}N/M=H^{G}M/M\cong H^{G}/(H^{G}\cap M) $ is a chief factor of $ G $. Since $ H $ satisfies the partial $ \LL $-$ \Pi $-property in $ G $, we have that $ | G : \mathrm{N}_{G} (H(H^{G}\cap M))| $ is a $ \pi (H(H^{G}\cap M)/(H^{G}\cap M)) $-number. Now $ \mathrm{N}_{G}(HM)=\mathrm{N}_{G}((H^{G}\cap HM)M)\geq \mathrm{N}_{G}(H^{G}\cap HM) $, we see that $ |\mathrm{N}_{G}(HM)|=|\mathrm{N}_{G}(H(H^{G}\cap M))| $. It follows that $ |G:\mathrm{N}_{G}(HM)| $ is a $ \pi (HM/M) $-number. This meas that $ HN $ satisfies the partial $ \LL $-$ \Pi $-property in $ G $, as wanted.
	
	(2) We  may assume that $ N\leq H $ by (1). Then $ H_{G}\geq N $. It is no loss to assume that $ H\ntrianglelefteq G $. Let $ L/N $ be a maximal $ G/N $-invariant subgroup of $ (H/N)^{G/N}=H^{G}/N $ with $  (H/N)_{G/N}=H_{G}/N\leq L/N $. Clearly,  $ H^{G} /L $ is a chief factor of $ G $ and $  H_{G}\leq L $. Since $ H $ satisfies the partial $ \LL $-$ \Pi $-property in $ G $, we have that  $ |G/L :\mathrm{N}_{G/L}(HL/L)| $ is a $ \pi (HL/L) $-number. Write $ \overline{G}=G/N $.  Hence $ |\overline{G}/\overline{L} :\mathrm{N}_{\overline{G}/\overline{L}} (\overline{HL}/\overline{L})| $ is a $ \pi (\overline{HL}/\overline{L}) $-number.  This means that $ \overline{H} $ satisfies the partial $ \LL $-$ \Pi $-property in $\overline{G}$.
\end{proof}

If $ P $ is either an odd order $ p $-group or a quaternion-free $ 2 $-group, then we use $ \Omega(P) $ to denote the subgroup $ \Omega_{1} (P) $.  Otherwise, $ \Omega (P) = \Omega_{2} (P) $.

\begin{lemma}\label{hypercenter}
	Let $ P $ be a normal $ p $-subgroup of a group $ G $ and $ D $ a Thompson critical subgroup of $ P $ (see \cite[page 186]{Gorenstein-1980}). If $ P/\Phi(P) \leq \mathrm{Z}_{\UU}(G/\Phi(P)) $ or  $ \Omega(D) \leq \mathrm{Z}_{\UU}(G) $, then $ P \leq  \mathrm{Z}_{\UU}(G) $.
\end{lemma}

\begin{proof}[\bf{Proof}]
	 By \cite[Lemma 2.8]{Chen-xiaoyu-2016}, the conclusion follows.
\end{proof}


\begin{lemma}[{\cite[Lemma 2.10]{Chen-xiaoyu-2016}}]\label{critical}
	Let $ D $ be a Thompson critical subgroup of a nontrivial $ p $-group $ P $.
	
	\begin{enumerate}[\rm(1)]
		\item If $ p > 2 $, then the exponent of $ \Omega_{1}(D) $ is $ p $.
		\item If $ P $ is an abelian $ 2 $-group, then the exponent of $ \Omega_{1}(D) $ is $ 2 $.
		\item If $ p = 2 $, then the exponent of $ \Omega_{2}(D) $ is at most $ 4 $.
	\end{enumerate}

\end{lemma}

\begin{lemma}[{\cite[Lemma 3.1]{Ward}}]\label{charcteristic}
	Let $ P $ be a nonabelian quaternion-free $ 2 $-group. Then $ P $ has a characteristic subgroup of index $ 2 $.
\end{lemma}

\begin{lemma}[{\cite[Lemma 2.10]{Su-2014}}]\label{phi}
	Let $ p $ be a prime, $ E $ a normal subgroup of a group $ G $ such that $ p $ divides the order of $ E $. Then $ E \leq  \mathrm{Z}_{\UU_{p}}(G) $
	if and only if $ E /\Phi(E) \leq \mathrm{Z}_{\UU_{p}}(G/\Phi(E)) $.
\end{lemma}

\begin{lemma}\label{also}
	Let $ p $ be a prime. Assume that $ N $ is a minimal normal subgroup of a group $ G $ and $ L\leq G $. Assume that $ |N|=|L|=p $. If $ LN $ satisfies the partial $ \LL $-$ \Pi $-property in $ G $, then $ L $ also satisfies the partial $ \LL $-$ \Pi $-property in $ G $.
\end{lemma}

\begin{proof}[\bf{Proof}]
	
	If $ L \unlhd G $, there is nothing to prove. Hence we can suppose that $ L \ntrianglelefteq G $. Thus $ N\not =L $ and $ |LN|=p^{2} $.  Let $ L^{G}/K $ be a chief factor of $ G $ with $ L_{G}\leq K $. Assume that $ (LNK/K)\cap (L^{G}/K) $ has order $ p^{2} $. Then $ LN\cap K=1 $ and $ LNK\leq L^{G} $. But $ L^{G}=NK $ and $ |(LNK/K)\cap (L^{G}/K )| = |(LNK/K ) \cap (NK/K)| = | NK/K| = p $. This contradicts our  assumption. Hence we have $ |(LNK\cap L^{G})/K|\leq p $. Obviously, if $ |LK/K|= 1 $, there is nothing to prove.  Thus we can assume that $ |LK/K| = p = |(LNK\cap L^{G})/K| $. Therefore $ LK=LNK\cap L^{G} $.
	
	It suffices to show that $ |G:\mathrm{N}_{G}(LK)| $  is a $ p $-number. If $ (LN)_{G}=LN $, then $ |G:\mathrm{N}_{G}(LK)|=|G:\mathrm{N}_{G}(LNK\cap L^{G})|=1 $  is a $ p $-number, as desired. Assume that $ (LN)_{G}=N $. If $ N\leq K $, then $ K $  is a maximal $ G $-invariant subgroup  of $ L^{G}N $. Since $ LN $ satisfies the partial $ \LL $-$ \Pi $-property in $ G $, we have that  $ |G:\mathrm{N}_{G}(LK)| $ is a $ p $-number, as desired. If $ N\nleq K $, then $ NK $ is a  maximal $ G $-invariant subgroup  of $ L^{G}N $ with $ (NL)_{G}\leq NK $. Since $ LN $ satisfies the partial $ \LL $-$ \Pi $-property in $ G $, we have that  $ |G:\mathrm{N}_{G}(LNK)| $ is a $ p $-number,  and so $ |G:\mathrm{N}_{G}(LK)|=|G:\mathrm{N}_{G}(LNK\cap L^{G})| $  is a $ p $-number, as desired. Hence the proof is complete.
\end{proof}



\begin{lemma}\label{Normal}
	Let  $ P\in {\rm Syl}_{p}(G) $, $ d $ be a power of $ p $ with $ p\leq d\leq |P| $, and let $ N $ be a minimal normal subgroup of $ G $ with $ d\big | |N| $. Assume  that every subgroup of $ P $ of order $ d $ satisfies the partial $ \Pi $-property in $ G $. Then $ |N|=d $; and if addition $ d\geq p^{2} $, then every minimal normal subgroup of $ G $ with order divisible by $ p $ is $ G $-isomorphic to $ N $.
\end{lemma}

\begin{proof}[\bf{Proof}]
	Let $ H $ be a normal subgroup of $ P $ of order $ d $ such that $ H \leq N $. Then $ H^{G}=N $. By hypothesis, $ H $ satisfies the partial $ \LL $-$ \Pi $-property in $ G $. Hence  $ |G : \mathrm{N}_{G}(H)| $ is a $ p $-number, and so $ H\unlhd G $. The minimality of $ N $ yields that $ H=N $. Thus $ N $ is elementary abelian  of order $ d $.
	
	Suppose  further that $ d\geq p^{2} $. Let $ K $ be a minimal normal subgroup of $ G $ with $ p\big| |K|$. Assume that $ N $ and $ K $ are not $ G $-isomorphic, and  we work to obtain a contradiction.  Let $ L $ be a normal subgroup of $ P $ of order $ d $ such that $ |L \cap N| = \frac{d}{p} $ and $ |L \cap K| = p $.

	Suppose that  $ |K|\not = p $. Clearly, $ L^{G}=N\times K $ and $ L_{G}=1 $. By hypothesis, $ L $ satisfies the partial $ \LL $-$ \Pi $-property in $ G $. For the $ G $-chief factor $ L^{G}/K $,  we have that $ |G:\mathrm{N}_{G}(LK)| $ is  a $ p $-number, and thus $ LK\unlhd G $. Hence  $ N\times K=LK $. This yields that $ |L\cap K|=1 $, a contradiction.
	For the $ G $-chief factor $ L^{G}/N $, $ |G:\mathrm{N}_{G}(LN)| $ is  a $ p $-number, and thus $ LN\unlhd G $. Hence $ N\times K=LN $. It follows that $ |K| = |L/(L \cap N)| = p $,  a contradiction.
	
	Now suppose that $ |K| = p $. Then $ K < L $, $ L_{G}=K $ and $ L^{G}=N\times K $. Since $ L  $ satisfies the partial $ \LL $-$ \Pi $-property in $ G $, we have that $ |G:\mathrm{N}_{G}(L)|=|G:\mathrm{N}_{G}(LK)| $ is a $ p $-number. Thus $ L\unlhd G $, also a contradiction.
\end{proof}

\begin{lemma}\label{unique}
	Let $ \FF $ be a saturated formation and $ N $ a normal subgroup of a group $ G $. If $ M \unlhd G $ is of minimal order such that $ M/(M \cap N)\nleq \mathrm{Z}_{\FF}(G/(M\cap N)) $, then $ M $ has a unique maximal $ G $-invariant subgroup, say $ L $. Moreover, $ L\geq M \cap N $, $ L/(M \cap N) \leq
	\mathrm{Z}_{\FF}(G/(M \cap N)) $ and $ M/L \nleq \mathrm{Z}_{\FF}(G/L) $.
\end{lemma}

\begin{proof}[\bf{Proof}]
	Assume that $ M $ has two different maximal $ G $-invariant subgroups,  $ L_{1} $, $ L_{2} $ say. By the minimality of $ M $, $ L_{i}/(L_{i}\cap N)\leq \mathrm{Z}_{\FF}(G/(L_{i}\cap N)) $  for each $ i \in \{ 1, 2 \} $. This implies that $ L_{1}N/N $ and $ L_{2}N/N $  are contained in $ \mathrm{Z}_{\FF}(G/N) $. Consequently $ M/(M \cap N)\cong MN/N = (L_{1}N/N)(L_{2}N/N) \leq \mathrm{Z}_{\FF}(G/N) $, a contradiction. Hence $ M $ has  a unique maximal $ G $-invariant subgroup,  $ L $ say. Clearly,  $ M > M \cap N $. Now the uniqueness of $ L $ implies $ L\geq M \cap N $. Since $ L/(L \cap N)\leq \mathrm{Z}_{\FF}(G/(L\cap N)) $  and $ L \cap N = M \cap N $, we deduce that $ L/(M \cap N)\leq \mathrm{Z}_{\FF}(G/(M\cap N)) $. Since $ M/(M \cap N)\nleq \mathrm{Z}_{\FF}(G/(M\cap N)) $, it follows from \cite[Theorem 1.2.7(b)]{Guo2015} that $ M/L\nleq \mathrm{Z}_{\FF}(G/L) $.
\end{proof}

\begin{lemma}[{\cite[Lemma 2.12]{Qiu}}]\label{hyper}
	Let $ N $ be a normal subgroup of a group $ G $ and $ P\in {\rm Syl}_{p}(N) $. Then $ N\leq \mathrm{Z}_{\UU_{p}}(G) $ if and  only if
	all cyclic subgroups of $ P $ of order $ p $ or $ 4 $ (when   $ P $ is  not quaternion-free) are contained in $ \mathrm{Z}_{\UU_{p}}(G) $.
\end{lemma}

\begin{lemma}[{\cite[Proposition 2.5]{Qian2012}}]\label{Qian}
	Let $ p $ be a prime and let $ N = W_{1} \times  \cdots \times W_{s}  $ be a normal  subgroup of $ G $ with $ \mathrm{C}_{G}(N) = 1 $, where $ W_{i} $  is a nonabelian simple group of order divisible by $ p $ for every $ i=1, ..., s $. Then $ |G/N|_{p} < |N|_{p}  $.
\end{lemma}

\begin{lemma}\label{p-group}
	Let $ H $ be a nonidentity $ p $-subgroup of a group $ G $ and $ N $ a mininal normal subgroup of $ G $. If $ H $ satisfies the partial $ \LL $-$ \Pi $-property in $ G $ and $ H\leq N $,  then $ N $ is a $ p $-group.
\end{lemma}

\begin{proof}[\bf{Proof}]
	Clearly, $ H^{G}=N $ and $ H_{G}=1 $. By hypothesis, $ H $ satisfies the partial  $ \LL $-$ \Pi $-property in $ G $. Hence  $ |G:\mathrm{N}_{G}(H)| $ is a $ p $-number. This implies
	that $ G = \mathrm{N}_{G}(H)P  $, where $ P $ is a Sylow $ p $-subgroup of $ G $ containing $ H $. By \cite[Lemma 3.4.9]{Guo-2000}, it follows that $ H\leq P_{G} $. Thus $ N=H^{G}\leq \mathrm{O}_{p}(G) $.
\end{proof}




\begin{lemma}\label{less}
	Let $ G $ be a $ p $-soluble group of order divisible by $ p $. Then $ |G/\mathrm{O}_{p',p}(G)|_{p} < |\mathrm{O}_{p',p}(G)|_{p} $ .
\end{lemma}
\begin{proof}[\bf{Proof}]
	By \cite[Corollary 1.9(iii)]{Wolf-1984}, the conclusion holds.
\end{proof}

\begin{lemma}[{\cite[Lemma 3.3(4)]{Qian-2020}}]\label{cyclic}
	Let $ H $ be a $ p' $-group and let $ V $ be a faithful and irreducible $ \mathbb{F}_{p}[H] $-module. Assume that $ \mathrm{dim}_{\mathbb{F}_{p}}V $ is a prime. Then $ H $ is cyclic if and only if $ V $ is not absolutely irreducible.
\end{lemma}

\section{Proofs}

In order to prove Theorem \ref{super}, we shall prove the following result at first.

\begin{theorem}\label{small}
	Let $ P $ be a normal $ p $-subgroup of a group $ G $ for some prime $ p\in \pi(G) $. If every cyclic subgroup of $ P $ of order $ p $ or $ 4 $ (when $ P $ is  not quaternion-free) satisfies the partial $\LL$-$ \Pi $-property in $ G $, then $ P\leq \mathrm{Z}_{\UU}(G) $.
\end{theorem}

\begin{proof}[\bf{Proof}]
	Assume that this theorem  is not true and let $ ( G, P ) $ be a counterexample for which $ |G|+|P| $ is minimal. We proceed via the following steps.
	
	\vskip0.1in
	
	\textbf{Step 1.} $ P $ has a unique  maximal $ G $-invariant subgroup, say $ N $. Moreover, $ N\leq \mathrm{Z}_{\UU}(G) $ and $ P/N $ is not cyclic.

	Let $ N $ be a maximal $ G $-invariant subgroup of $ P $.
	Clearly, $ (G, N) $ satisfies the hypotheses of the theorem. By the choice of $ (G, P) $, we have that $ N \leq \mathrm{Z}_{\UU}(G) $.  If $ P/N $ is cyclic, then $ P/N \leq \mathrm{Z}_{\UU}(G) $,  and so $ P \leq
	\mathrm{Z}_{\UU}(G) $, a contradiction. Hence $ P/N $ is noncyclic. Now assume that $ T $ is a  maximal $ G $-invariant subgroup which is different from $ N $. Then $ T\leq \mathrm{Z}_{\UU}(G) $ by the same argument as above. By $ G $-isomorphism $ P/N\cong TN/N\cong T/(T\cap
	N) $, we deduce that $ P/N\leq \mathrm{Z}_{\UU}(G/N) $, a contradiction. Hence $ N $ is the unique  maximal $ G $-invariant subgroup of $ P $.
	
	\vskip0.1in
	
	\textbf{Step 2.} The exponent of $ P $ is $ p $ or $ 4 $ (when $ P $ is  not quaternion-free).

	Let $ D $ be a Thompson critical subgroup of $ P $.
	If $ \Omega(D)< P $, then $ \Omega(D)\leq N $ since $ N $ is the unique maximal $ G $-invariant subgroup of $ P $. In view of Step 1, we have $ \Omega(D)\leq \mathrm{Z}_{\UU}(G) $. By Lemma \ref{hypercenter}, it follows that $ P\leq \mathrm{Z}_{\UU}(G) $, a contradiction. Thus $ P = D = \Omega (D) $. If $ P $ is a nonabelian quaternion-free $ 2 $-group, then $ P $ has a characteristic subgroup $ K $ of index $ 2 $ by Lemma \ref{charcteristic}. By Step 1, we have  that $ K=N \leq \mathrm{Z}_{\UU}(G) $, and hence  $ P\leq \mathrm{Z}_{\UU}(G) $, which is impossible. This means that $ P $ is a nonabelian $ 2 $-group if and only if $ P $ is not quaternion-free. By  Lemma \ref{critical}, the exponent of $ P $ is $ p $ or $ 4 $ (when $ P $  is not quaternion-free). Hence Step 2 holds.
	
	\vskip0.1in
	
	\textbf{Step 3.} The final contradiction.

	Let $ G_{p} $ be a Sylow $ p $-subgroup of $ G $. Note that $ P/N \cap
	\mathrm{Z}(G_{p}/N) > 1 $. Let $ X/N $ be a subgroup of $ P/N \cap
	\mathrm{Z}(G_{p}/N) $ of order $ p $. Then we may choose an element $ x\in X\setminus N $. Write $ H=\langle x \rangle $. Then $ X=HN $ and $ H $ is a subgroup of order $ p $ or $ 4 $ (when $ P $ is not quaternion-free) by Step 2. In view of  Step 1, we know that $ H^{G}=P $ and $ H_{G}\leq N $. By hypothesis, $ H $ satisfies the partial $ \LL $-$ \Pi $-property in $ G $. Then $ |G:\mathrm{N}_{G}(X)|=|G:\mathrm{N}_{G}(HN)| $ is a $ p $-number. Since $ X\unlhd G_{p} $, we see that $ X\unlhd G $. This forces that $ P/N=X/N=HN/N\cong H/(H\cap N) $ is cyclic, which contradicts Step 1. This final contradiction completes the proof.
\end{proof}

Now we begin to prove Theorem \ref{super}.

\begin{proof}[\bf{Proof of Theorem \ref{super}}]
	Assume that this theorem is not true and let $ (G, E) $ be a minimal counterexample for which $ |G|+|E| $ is minimal.
	
	\vskip0.1in
	
	\textbf{Step 1.} $ \mathrm{O}_{p'}(E)=1 $.

	By Lemma \ref{OV}(2), the hypotheses hold for $ (G/\mathrm{O}_{p'}(E), N/\mathrm{O}_{p'}(E)) $. Thus the minimal choice of $ (G, E) $ implies that $ \mathrm{O}_{p'}(E) = 1 $.
	
	\vskip0.1in
	
	\textbf{Step 2.} $ E $ is not $ p $-soluble.

	Assume that $ E $ is $ p $-soluble, and we work to obtain a contradiction. Since $ \mathrm{O}_{p'}(E)=1 $, it follows from \cite[Lemma 2.10]{ball-2009} that $ \mathrm{F}^{*}_{p}(E) = \mathrm{F}_{p}(E) =
	\mathrm{O}_{p}(E) $ .  By Theorem \ref{small}, we have $ \mathrm{F}^{*}_{p}(E) =\mathrm{O}_{p}(E) \leq \mathrm{Z}_{\UU}(G) \leq \mathrm{Z}_{\UU_{p}}(G) $. Hence $ N\leq \mathrm{Z}_{\UU_{p}}(G) $ by \cite[Lemma 2.13]{Su-2014}, a contradiction.
	
	\vskip0.1in
	
	\textbf{Step 3.} $ E $ has a unique maximal $ G $-invariant subgroup $ U $. Moreover, $ 1<U\leq \mathrm{Z}_{\UU_{p}}(G) $.

	Let  $ H $ be a cyclic subgroup of $ P $ of order $ p $ or $ 4 $ (when $ P $ is not quaternion-free).
	Assume that  $ E $ is a minimal normal subgroup of $ G $. Then $ H^{G}=E $. 	By Lemma \ref{p-group}, $ E $ is a $ p $-group, contrary to Step 2. Hence $ E $ has a maximal $ G $-invariant subgroup $ U>1 $. It is easy to see  that the hypotheses also hold for $ (G, U) $ and the minimal choice of $ (G, E) $ yields that $ U\leq  \mathrm{Z}_{\UU_{p}}(G) $. Now assume that  $ R $ is a maximal $ G $-invariant subgroup of $ E $, which is different from $ U $. Then $ R\leq \mathrm{Z}_{\UU_{p}}(G) $ with the same argument
	as above, and thus $ E=UR\leq \mathrm{Z}_{\UU_{p}}(G) $, a contradiction. Therefore, $ U $ is the unique maximal $ G $-invariant subgroup of $ E $. 
	
	\vskip0.1in
	
	\textbf{Step 4.}  Every cyclic subgroup of $ P $ of order $ p $ or $ 4 $ (when $ P $ is not quaternion-free)  is contained in $ U $.

	Assume that there exists a cyclic subgroup $ L $ of $ P $ of order $ p $ or $ 4 $ (when $ P $ is not quaternion-free)  such that $ L\nleq U $. The uniqueness of $ U $ yields that $ L_{G}\leq U\leq E=L^{G} $, and so $ (LU/U)^{G/U}=E/U $. By hypothesis, $ L $ satisfies the partial $ \LL $-$ \Pi $-property in $ G $. Applying  Lemma \ref{OV}(2), $ LU/U $ satisfies the partial $ \LL $-$ \Pi $-property in $ G/U $. By Lemma \ref{p-group}, it turns out that $ E/ U $ is a $ p $-group.  Hence  $ E $ is $ p $-soluble by Step 3, a contradiction.   Therefore, Step 4 holds.
	
	\vskip0.1in
	
	\textbf{Step 5.}  The final contradiction.

	Since $ U\leq \mathrm{Z}_{\UU_{p}}(G) $, it follows from Step 4 and Lemma \ref{hyper} that $ E\leq \mathrm{Z}_{\UU_{p}}(G) $. This  final contradiction completes the proof.
\end{proof}

\begin{lemma}\label{ele}
	Let $ P \in {\rm Syl}_{p}(G) $ and $ d $ be a  power of $ p $ such that $ p^{2} \leq d < |P| $. Assume that $ \mathrm{O}_{p'}(G)=1 $ and every subgroup of $ P $ of order $ d $ satisfies the partial $ \LL $-$ \Pi $-property in $ G $. If $ G $ has a minimal normal subgroup of order $ d $, then
	
	\begin{enumerate}[\rm(1)]
		\item $ G/\mathrm{Soc}(G) $ is $ p $-supersoluble;
		\item $ \Phi(G) = 1 $ and $ \mathrm{Soc}(G) =\mathrm{O}_{p}(G) $ is a homogeneous $ \mathbb{F}_{p}[G] $-module;
		\item If in addition $ |\mathrm{O}_{p}(G)|>d $, then every irreducible $ \mathbb{F}_{p}[G] $-submodule of $ \mathrm{O}_{p}(G) $ is not absolutely irreducible.
	\end{enumerate}

\end{lemma}

\begin{proof}[\bf{Proof}]
	 By Lemma \ref{Normal}, all minimal normal subgroups of $ G $ are $ G $-isomorphic and elementary abelian of order $ d $. Clearly, $ \mathrm{Soc}(G) $ is a completely reducible $ \mathbb{F}_{p}[G] $-module. Thus $ \mathrm{Soc}(G)\leq \mathrm{O}_{p}(G) $ is a homogeneous $ \mathbb{F}_{p}[G] $-module. Write $ \mathrm{Soc}(G)=N $.
	
	 (1)  Suppose that the result is false and $ G $ is a counterexample of minimal order. At first, we will show that every subgroup of $ P/N $ of order $ p $ satisfies  the partial $ \LL $-$ \Pi $-property in $ G/N $.  Let $ X/N $ be an arbitrary  subgroup of $ P/N $ of order $ p $. Then $ X $ is not a cyclic subgroup
	 since $ N $ is not cyclic. It follows that $ X $ has another subgroup $ X_{1} $ of order $ d $ such that $ X = X_{1}N $. By hypothesis, $ X_{1} $ satisfies the partial $ \LL $-$ \Pi $-property in $ G $. Applying Lemma \ref{OV}(2), $ X/N=X_{1}N/N $ satisfies the partial $ \LL $-$ \Pi $-property in $ G/N $. Hence  every subgroup of $ P/N $ of order $ p $ satisfies the partial $ \LL $-$ \Pi $-property in $ G/N $.
	
	 If $ p>2 $ or $ p=2 $ and $ P/N $ is quaternion-free, then $ G/N $ is $ p $-supersoluble according to Theorem \ref{super}. So we may suppose that $ p=2 $ and $ P/N $ is not quaternion-free.
	
	 Let $ T \unlhd G $ be of minimal order such that $ T/(T\cap N) \nleq  \mathrm{Z}_{\UU_{2}}(G/(T\cap N)) $. Clearly, $ T\cap N $ is the product of all minimal $ G $-invariant subgroups of $ T $. In particular,  $ T\cap N= V_{1} \times \cdots \times V_{s}$, $ s\geq  1 $, where  $ V_{i} $ is a minimal $ G $-invariant subgroup of $ T $ and $ V_{i} $  is  elementary abelian  of order $ d $ for all $ 1\leq i\leq s $.
	
	 Set $ N_{0}=T\cap N $. By Lemma \ref{unique}, $ T $ has a unique maximal $ G $-invariant subgroup, say $ U $. Moreover,  $ U \geq N_{0} $, $ U/N_{0} \leq
	 \mathrm{Z}_{\UU_{2}}(G/N_{0}) $, and $ T/U \nleq \mathrm{Z}_{\UU_{2}}(G/U) $. Notice that $ U/N_{0}=T/N_{0}\cap \mathrm{Z}_{\UU_{2}}(G/N_{0}) $.  By Lemma \ref{hyper}, $ (P\cap T)/N_{0} $ possesses a cyclic subgroup $ \langle a \rangle N_{0}/N_{0} $ of order $ 2 $ or $ 4 $ (when $ (P\cap
	 T)/N_{0} $ is  not quaternion-free) such that $ a \in T \setminus U $.  Let $ a_{0} \in T\setminus U  $ be a $ 2 $-element such that $ o(a_{0})  $  is as small as possible. Since $ o(a_{0}) \leq o(a) $ and $ N_{0} $ is elementary abelian, we have $ o(a_{0})\leq 8 $. The minimality of $ o(a_{0}) $ implies that $ a_{0}^{2}\in U $.
	
	 We shall complete the proof in the following two cases.
	
	 \vskip0.1in
	 \textbf{Case 1.} $ o(a_{0})> d $.

	 In this case, $ o(a_{0}) = 8 $ and $ |N_{0}| = d=4 $.  Suppose that $ N_{0} \nleq  \mathrm{Z}(T) $.  The uniqueness of $ U $ yields $ \mathrm{C}_{T}(N_{0}) \leq U $.  Therefore $ T/\mathrm{C}_{T}(N_{0})\lesssim {\rm Aut}(N_{0})\cong \mathrm{Sym}(3) $. It follows that $ T/U\leq \mathrm{Z}_{\UU_{2}}(G/U) $, and thus $ T/N_{0}\leq \mathrm{Z}_{\UU_{2}}(G/N_{0}) $, a contradiction. Hence $ N_{0} \leq  \mathrm{Z}(T) $. By \cite[Lemma 2.5.25]{Guo2015}, all  $ 2 $-supersoluble groups  are  $ 2 $-nilpotent.  Therefore, $ U/N_{0} $ is $ 2 $-nilpotent. Observe that $ N_{0}\leq \mathrm{Z}(T)\cap U\leq \mathrm{Z}(U) $. Since $ \mathrm{O}_{2'}(G)=1 $, we deduce that $ U $ is a  $ 2 $-group. Since $ a_{0}^{2}\in U $ and $ o(a_{0}^{2})=4 $, we have $ \Phi(U)>1 $. Let $ L $ be a minimal normal subgroup of $ G $ contained in $ \Phi(U) $.  Note that $ N_{0} $ is a completely reducible $ \mathbb{F}_{2}[G] $-module. Without loss of generality, we may assume that $ L = V_{1} $. Write $ \widetilde{G}  = G/(V_{2} \times \cdots \times V_{s}) $. Observe that  $ \widetilde{N_{0}}\leq \widetilde{\Phi(U)} \leq \Phi(\widetilde{U}) $ and $ \mathrm{Z}_{\UU_{2}}(\widetilde{U}/\widetilde{N_{0}})\leq \mathrm{Z}_{\UU_{2}}(\widetilde{G}/\widetilde{N_{0}}) $. By Lemma \ref{phi}, it follows that  $\widetilde{U}\leq \mathrm{Z}_{\UU_{2}}(\widetilde{G}) $, and thus $ V_{1}\cong \widetilde{N_{0}}\cong \mathrm{C}_{2} $, a contradiction.

	
	 \vskip0.1in
	 \textbf{Case 2.} $ o(a_{0})\leq d $.
	
	 Let $ H $ be a subgroup of $ P $ of  order $ d $ such that $ \langle a_{0} \rangle\leq H\leq \langle a_{0} \rangle N_{0} $. Since $ U $ is the unique maximal $ G $-invariant subgroup of $ T $ and $ a_{0}\in T\setminus U $, we can deduce that $ H^{G}=T $ and $ H_{G}\leq U $.
	 By hypothesis, $ H $ satisfies the partial $ \LL $-$ \Pi $-property in $ G $. Hence $ |G:N_{G}(HU)| $ is a $ 2 $-number. It follows that $ T=H^{G}=(HU)^{G}=H^{P}U $, and so $ T/U $ is a $ 2 $-group.
	
	 Note that  $ U/N_{0} $ is $ 2 $-nilpotent. Let $ L/N_{0} $ be a normal $ 2 $-complement of $ U/N_{0} $. Set $ \overline{G}=G/L $. We will show that the exponent of $ \overline{T} $ is $ 2 $ or $ 4 $ (when $ \overline{T} $  is not quaternion-free).   Let $ \overline{D} $ be a Thompson critical subgroup of $ \overline{T} $.  If $ \Omega(\overline{T})< \overline{T} $, then $ \Omega(\overline{T})\leq \overline{U} $ by the  uniqueness of $ U $. Thus  $ \Omega(\overline{T})\leq \mathrm{Z}_{\UU}(\overline{G}) $.
	 By Lemma \ref{hypercenter}, we have that $ \overline{T}\leq \mathrm{Z}_{\UU}(\overline{G}) $, a  contradiction. Hence  $ \overline{T} = \overline{D} = \Omega(\overline{D}) $.  If $ \overline{T} $ is a nonabelian quaternion-free $ 2 $-group, then $ \overline{T} $ has a characteristic subgroup $ \overline{T_{0}} $ of index $ 2 $ by Lemma \ref{charcteristic}. The uniqueness of $ U $ implies $ T_{0}=U $. Thus  $ \overline{T} \leq \mathrm{Z}_{\UU}(\overline{G}) $, which is impossible. This means that $ \overline{T} $ is a nonabelian $ 2 $-group if and only if $ \overline{T} $ is not quaternion-free.  In view of  Lemma \ref{critical}, the exponent of $ \overline{T} $ is $ 2 $ or $ 4 $ (when $ \overline{T} $  is not quaternion-free), as desired.
	

	 Since $ N_{0} $ is elementary abelian, we see that every element of $ P\cap T $ has order at most $ 8 $. Clearly, $ |T/U|\geq 4 $. There exists  a normal subgroup $ Q $ of $ P $  such that $ P\cap U< Q< P\cap T $ and $ |Q:(P\cap U)|=2 $.  Pick a nonidentity  element $ b\in Q\setminus (P\cap U) $. Then $ o(b)\leq 8 $ and $ b^{2} \in (P\cap U) $.
	

	 If $ o(b) >d $, then  $ o(b) = 8 $ and $ |N_{0}|=d=4 $. With a similar argument as Case 1, we have  $ U\leq \mathrm{Z}_{\UU_{2}}(G) $, and thus $ |N_{0}|=2 $, a contradiction.
	

	
	 Assume that $ o(b) \leq d $. Let $ M $ be a  subgroup of $ Q $ of  order  $ d $ such that $ \langle b \rangle \leq M $. If $ b\in U $, then  $ P\cap U<\langle b \rangle(P\cap U)\leq U $, a contradiction. Therefore $ b
	 \not \in U $. The  uniqueness of $ U $ implies that $ M^{G}=T $ and $ M_{G}\leq U $. Then $ Q=MU $. By hypothesis, $ M $ satisfies the partial $ \LL  $-$ \Pi $-property in $ G $. Then  $ |G:N_{G}(Q)|=|G:N_{G}(MU)| $ is a $ 2 $-number. Thus $ Q\unlhd G $, which contradicts the fact that $ U $ is the  unique maximal $ G $-invariant subgroup of $ T $.

	(2) Assume that $ \Phi(G) > 1 $. Let $ R $ be a minimal $ G $-invariant subgroup of $ \Phi(G) $. Then $ R\leq N=\mathrm{Soc}(G) $. Since $ N $ is a completely reducible $ \mathbb{F}_{p}[G] $-module, there exists $ A\unlhd G $ such
	that $ N = R \times A $.  Observe that $ N/A\leq \Phi(G)A/A \leq \Phi(G/A) $ and $ (G/A)/(N/A) $ is $ p $-supersoluble by (1). It follows by Lemma \ref{phi} that $ G/A $ is $ p $-supersoluble, and thus $ R\cong N/A\cong \mathrm{C}_{p} $, a contradiction. As a consequence, $ \Phi(G) = 1 $. Hence $ \mathrm{O}_{p}(G) $ is the product of all minimal normal subgroups of $ G $. This implies $ N=\mathrm{O}_{p}(G) $, and so (2) holds.
	
	(3) Note that all minimal normal subgroups of $ G $ are $ G $-isomorphic and elementary abelian of order $ d $. Let $ X $ and $ Y $ be two isomorphic irreducible $ \mathbb{F}_{p}[G] $-modules of $ N=\mathrm{Soc}(G)=\mathrm{O}_{p}(G) $. Then there exists a $ G $-isomorphism $ f : X \mapsto Y $. Suppose that $ X $ is absolutely irreducible. Then $ \mathrm{End}_{\mathbb{F}_{p}[H]}(X) \cong\mathbb{F}_{p}  $. According to \cite[Proposition B.8.2]{MR1169099}, the number of the minimal $ G $-invariant subgroups of $ X\times Y $ is $ \frac{p^{2}-1}{p-1}=p+1 $. Thus the minimal $ G $-invariant subgroups of $ X\times Y $ are exactly
	
	\begin{center}
		$ X $ and $ Y_{m} = \{x^{m}f(x)|x\in X \}$, where $ m = 0, 1, ..., p-1 $.
	\end{center}
	
	\noindent Let $ x_{1} $ be an element of $ X\cap \mathrm{Z}(P) $ of order $ p $. Then $ \langle f(x_{1})\rangle \unlhd P $. Let $ B $ be a subgroup of $ X\times Y $ of order $ d $ such that $ \langle x_{1} \rangle \langle f(x_{1})\rangle \subseteq B\unlhd P $. Clearly, $ B_{G}=1 $ and $ B^{G}=XY $.  Let $ K $ be a maximal $ G $-invariant subgroup of $ B^{G}=XY $. Then $ |K|=d $. By hypothesis, $ B $  satisfies the partial $ \LL $-$ \Pi $-property in $ G $. Then $ |G:N_{G}(BK)| $ is a $ p $-number. Thus $ BK\unlhd G $. As a consequence, $ BK=XY $. This forces that $ B\cap K=1 $.
	However,  $ 1\not =x_{1}\in B\cap X $ and $ 1\not=x_{1}^{m}f(x_{1})\in B \cap Y_{m} $ for $ m = 0, 1, ..., p-1 $. This  contradiction yields that $ X $ is not absolutely irreducible. Hence (3) holds.
\end{proof}

\begin{lemma}\label{leq}
	Let $ P \in {\rm Syl}_{p}(G) $ and $ d $ be a  power of $ p $ such that $ p^{2} \leq d < |P| $. Assume that $ \mathrm{O}_{p'}(G)=1 $ and every subgroup of $ P $ of order $ d $ satisfies the partial $ \LL $-$ \Pi $-property in $ G $. If $ G $ has a minimal normal subgroup $ N $ of order $ \frac{d}{p} $, then
	
	\begin{enumerate}[\rm(1)]
		\item $ \mathrm{Soc}(G)\leq \mathrm{O}_{p}(G) $ and $ G/N $ is $ p $-supersoluble;
		\item In addition, if $ G $ is not $p$-supersoluble, then $ |\mathrm{Soc}(G)|\leq d $.
	\end{enumerate}

\end{lemma}

\begin{proof}[\bf{Proof}]
	By Lemma \ref{OV}(2), every subgroup of $ P/N $ of order $ p $ satisfies the partial $ \LL $-$ \Pi $-property in $ G/N $. Applying Lemma \ref{Normal}, we see that every minimal normal subgroup of $ G/N $ of order divisible by $ p $ has order $ p $.  In particular, every minimal normal subgroup of $ G $
	different from $ N $ has order $ p $, and hence  $ \mathrm{Soc}(G) \leq \mathrm{O}_{p}(G) $.
	
	(1) Now $ \mathrm{Soc}(G)\leq \mathrm{O}_{p}(G) $ has been proved. Assume that $ G/N $ is not $ p $-supersoluble, and we work to obtain a contradiction. Note that every subgroup of $ P/N $ of order $ p $ satisfies the partial $ \LL $-$ \Pi $-property in $ G/N $. If $ p>2 $ or $ p=2 $ and $ P/N $ is  quaternion-free, then $ G/N $ is $ p $-supersoluble by Theorem \ref{super}. Therefore, we may assume that $ p=2 $ and $ P/N $  is not quaternion-free. In particular, every minimal normal subgroup of $ G $ different from $ N $ has order $ 2 $, and thus lies in the center of $ G $.
	
	Let $ T\unlhd G $ be of minimal order such that $ T/(T\cap N)\nleq \mathrm{Z}_{\UU_{2}}(G/(T\cap N)) $. By Lemma \ref{unique}, $ T $ has a unique maximal $ G $-invariant subgroup, say $ U $. Moreover,  $ T\cap N\leq U $, $ U/(T\cap N)\leq \mathrm{Z}_{\UU_{2}}(G/(T\cap N)) $ and $ T/U\nleq \mathrm{Z}_{\UU_{2}}(G/U) $.  We claim that there exists a nonidentity $ 2 $-element $ a \in T\setminus U $ such that $ o(a) \leq d $. If $ d=4 $, then $ T\cap N\leq N\cong \mathrm{C}_{2} $, and thus $ U\leq \mathrm{Z}_{\UU_{2}}(G) $. Since $ T/U\nleq \mathrm{Z}_{\UU_{2}}(G/U) $, it follows from Lemma \ref{hyper} that the claim holds. Now suppose that $ d \geq 8 $. Applying Theorem \ref{hyper} to $ T/(T\cap N) $, there exists a $ 2 $-element $ a \in T\setminus U $ such that $ a^{2}\in T\cap N $ or $ a^{4}\in T\cap N $. Since $ N $ is elementary abelian, we see that $ o(a)\leq 8\leq d $, as claimed.
	
	
	Let $ H $ be a subgroup of $ P $ of  order $ d $ such that $ \langle a \rangle\leq H\leq \langle a \rangle N $. Since $ a\in T\setminus U $, the uniqueness of $ U $ yields that $ H^{G}=T $ and $ H_{G}\leq U $.
	By hypothesis, $ H $ satisfies the partial $ \LL $-$ \Pi $-property in $ G $. Hence $ |G:\mathrm{N}_{G}(HU)| $ is a $ 2 $-number. It follows that $ T=H^{G}=(HU)^{G}=H^{P}U $, and so $ T/U $ is a $ 2 $-group.
	
	Note that all $ 2 $-supersoluble groups are $ 2 $-nilpotent. Hence $ U/(T\cap N) $ has a normal $ 2 $-complement, say $ L/(T\cap N) $. Set $ \overline{G}=G/L $. We will show that the exponent of $ \overline{T} $ is $ 2 $ or $ 4 $ (when $ \overline{T} $  is not quaternion-free).   Let $ \overline{D} $ be a Thompson critical subgroup of $ \overline{T} $.  If $ \Omega(\overline{T})< \overline{T} $, then $ \Omega(\overline{T})\leq \overline{U} $ by the  uniqueness of $ U $. Thus  $ \Omega(\overline{T})\leq \mathrm{Z}_{\UU}(\overline{G}) $.
	By Lemma \ref{hypercenter}, we have that $ \overline{T}\leq \mathrm{Z}_{\UU}(\overline{G}) $, a  contradiction. Hence  $ \overline{T} = \overline{D} = \Omega(\overline{D}) $.  If $ \overline{T} $ is a nonabelian quaternion-free $ 2 $-group, then $ \overline{T} $ has a characteristic subgroup $ \overline{T_{0}} $ of index $ 2 $ by Lemma \ref{charcteristic}. The uniqueness of $ U $ implies $ T_{0}=U $. Thus  $ \overline{T} \leq \mathrm{Z}_{\UU}(\overline{G}) $, which is impossible. This means that $ \overline{T} $ is a nonabelian $ 2 $-group if and only if $ \overline{T} $ is not quaternion-free.  In view of  Lemma \ref{critical}, the exponent of $ \overline{T} $ is $ 2 $ or $ 4 $ (when $ \overline{T} $  is not quaternion-free), as desired.
	

	Since $ N $ is elementary abelian, we see that every element of $ P\cap T $ has order at most $ 8 $. Clearly, $ |T/U|\geq 4 $. There exists  a normal subgroup $ Q $ of $ P $  such that $ P\cap U< Q< P\cap T $ and $ |Q:(P\cap U)|=2 $.  Pick a nonidentity  element $ b\in Q\setminus (P\cap U) $. Then $ o(b)\leq 8 $ and $ b^{2} \in (P\cap U) $.
	


	
	Assume that $ o(b) \leq d $. Let $ M $ be a  subgroup of $ Q $ of  order  $ d $ such that $ \langle b \rangle \leq M $. If $ b\in U $, then  $ P\cap U<\langle b \rangle(P\cap U)\leq U $, a contradiction. Therefore $ b
	\not \in U $. The  uniqueness of $ U $ implies that $ M^{G}=T $ and $ M_{G}\leq U $. Then $ Q=MU $. By hypothesis, $ M $ satisfies the partial $ \LL  $-$ \Pi $-property in $ G $. Then  $ |G:N_{G}(Q)|=|G:N_{G}(MU)| $ is a $ 2 $-number. Thus $ Q\unlhd G $, which contradicts the fact that $ U $ is the  unique maximal $ G $-invariant subgroup of $ T $.
	
	Assume that $ o(b) >d $.  Then $ d=4 $ or $ 2 $. If $ d=2 $, then $ N=1 $, a contradiction. Therefore $ d=4 $, $ o(b)=8 $ and $ |N|=2 $. By hypothesis,  every subgroup of $ P $ of order $ 4 $ satisfies the partial $ \LL $-$ \Pi $-property in $ G $. Now we claim that  every subgroup $ X $ of $ P $ of order
	$ 2 $ satisfies the partial $ \LL $-$ \Pi $-property in $ G $.  It is no loss to assume that $ N\not =X $. Then $ XN $ has order $ 4 $. So $ XN $ satisfies the partial $ \LL $-$ \Pi $-property in $ G $. Applying Lemma  \ref{also}, we deduce that $ X $ satisfies the partial $ \LL $-$ \Pi $-property in $ G $, as claimed. It follows from Theorem  \ref{super} that $ G $ is $ 2 $-supersoluble, a contradiction. Hence (1) holds.
	
    (2) If $ N $ is cyclic, then $ G $ is $ p $-supersoluble by (1), a contadiction. Hence $ N $ is not cyclic. Consequently, $ d \geq p^{3} $. Note that every minimal normal subgroup of $ G $ different from $ N $ has order $ p $. So $ N $ is the unique minimal normal subgroup of $ G $ of order $ \frac{d}{p} $. If $ N $ is the unique minimal normal subgroup of $ G $, then $ |\mathrm{Soc}(G)|=|N|\leq d $, as wanted. Hence we can assume that $ G $ has a minimal normal subgroup $ R $ which is different from $ N $. Then $ |R|=p $ and $ \frac{d}{|R|}\geq p^{2} $. By Lemma \ref{OV}(2), every subgroup of $ P/R $ of order $ \frac{d}{|R|} $  satisfies the partial $ \LL $-$ \Pi $-property in $ G/R $. Note that $ NR/R $ is a minimal normal subgroup of $ G/R $ of order $ \frac{d}{p} $. It follows from Lemma \ref{Normal} that every minimal normal $ p $-subgroup of $ G/R $ has order $ \frac{d}{p} $.

    Let $ W $ be a minimal normal subgroup of $ G $. Then $ W\leq \mathrm{O}_{p}(G) $. If $ W\not =R $, then $ WR/R $ is a minimal normal $ p $-subgroup of $ G/R $. Thus $ |W|=\frac{d}{p} $, and so $ W=N $. This shows that $ N $ and $ R $ are all minimal normal subgroups of $ G $. Hence $ |\mathrm{Soc}(G)|=|NR|\leq d $, as wanted.
\end{proof}

\begin{theorem}\label{p-soluble}
	Let $ N $ be the largest normal $ p $-soluble subgroup of a group $ G $. 	Let $ P \in {\rm Syl}_{p}(G) $ and let $ d $ be a  power of $ p $ such that $ p^{2} \leq d \leq |P| $. Assume that every subgroup of $ P $ of order $ d $ satisfies the partial $ \LL $-$ \Pi $-property in $ G $. If $ G $ is not $ p $-soluble, then $$ \log_pd \geq {\rm max} \{3+\log_p |N|, \frac{1}{2}(3+\log_p |P|) \} .$$
\end{theorem}

\begin{proof}[\bf{Proof}]
	We proceed by induction on $ |G| $. By Lemma \ref{OV}(2), we may assume that $ \mathrm{O}_{p'}(G)=1 $. Suppose that $ N>1 $. Let $ T $ be a minimal normal subgroup of $ G $ contained in $ N $. Then $ T $ is a $ p $-group. By Lemma \ref{Normal}, we have $ |T|\leq d $. If $ |T|=d $ or $ \frac{d}{p} $, then $ G $ is $ p $-soluble by Lemmas \ref{ele} and \ref{leq}, a contradiction.
	Hence we may assume that $ |T|\leq \frac{d}{p^{2}} $. As a consequence, $ p^{2}\leq \frac{d}{|T|}< \frac{|P|}{|T|} $. By Lemma \ref{OV}(2), every subgroup of $ P/T $ of order $ \frac{d}{|T|} $ satisfies the partial $ \LL $-$ \Pi $-property in $ G/T $. Clearly, $ G/T $ is  not $ p $-soluble and $ N/T $ is the largest normal $ p $-soluble subgroup of $ G/T $.  By induction, we can get that $ \log_p \frac{d}{|T|}\geq {\rm max} \{3+\log_p |N/T|, \frac{1}{2}(3+\log_p |P/T|) \} $. Consequently, $ \log_p d\geq {\rm max} \{3+\log_p |N|, \frac{1}{2}(3+\log_p |P|) \} $, as desired.
	
	Assume that $ N=1 $. Clearly, $ \mathrm{Soc}(G)=M_{1}\times M_{2}\times\cdots\times M_{s} $, where $ M_{i} $ is a  nonabelian simple group of order divisible by $ p $ for any $ i=1, 2, ..., s $. Assume that $ |P\cap \mathrm{Soc}(G)|\geq d $. Let $ H $ be a subgroup of $ P\cap \mathrm{Soc}(G) $ of order $ d $ such that $ H\unlhd P $. Clearly, $ H_{G}=1 $. By hypothesis, $ H $ satisfies the partial $ \LL $-$ \Pi $-property in $ G $. Let $ K $ be a maximal $ G $-invariant subgroup of $ H^{G} $. Then $ |G:\mathrm{N}_{G}(HK)| $ is a $ p $-number, and so $ HK\unlhd G $ and $ H^{G}=HK $. This implies that $ H^{G}/K $ is a $ p $-group, contrary to the fact that $ H^{G}/K $ is isomorphic to some $ M_{j} $ where $ j\in \{1, 2, ..., s \} $. Therefore, $ |P\cap \mathrm{Soc}(G)|<d $, i.e., $ \log_p |P\cap \mathrm{Soc}(G)|\leq -1+ \log_pd $. Note that $ \mathrm{C}_{G}(\mathrm{Soc}(G))=1 $. By Lemma \ref{Qian}, we have $ |P\cap \mathrm{Soc}(G)|>\sqrt{|P|} $. Thus $ \log_pd\geq \frac{1}{2}(3+\log_p|P|) $. Since $ \log_pd \leq \log_p|P| $,  the above inequality yields $ \log_pd\geq \frac{1}{2}(3+\log_pd) $, i.e.,
	$ \log_pd \geq 3 = 3 + \log_p |N| $, and we are done.
\end{proof}


\begin{theorem}\label{dividing}
	Let $ G = P\rtimes H $, where $ H < G $ and $ P $ is a direct product of some minimal normal $ p $-subgroups of $ G $. Let $ d $ be a power of $ p $ such that $ p^{2}\leq d<P $. Assume that every subgroup of $ P $ of order $ d $ satisfies the partial $ \LL $-$ \Pi $-property in $ G $. If $ P\nleq \mathrm{Z}_{\UU}(G) $, then $ P $ is a homogeneous $ \mathbb{F}_{p}[H] $-module; and every irreducible  $ \mathbb{F}_{p}[H] $-submodule is not absolutely irreducible and has dimension dividing $ \log_pd $.
\end{theorem}

\begin{proof}[\bf{Proof}]
	We proceed by induction on $ |G| $. By Lemma \ref{OV}(2), $ G/\mathrm{O}_{p'}(G) $ inherits the hypotheses. So we can assume that $ \mathrm{O}_{p'}(G)=1 $. Let $ T $ be a minimal normal subgroup of $ G $ contained in $ P $. By Lemma \ref{Normal}, we have $ |T|\leq d $. If $ |T|=d $, then by Lemma \ref{ele}, the conclusions hold. If $ |T|=\frac{d}{p} $, then $ |P|\leq d $ by Lemma \ref{leq}(2), a contradiction. Hence every minimal normal subgroup of $ G $ contained in $ P $ has order at most $ \frac{d}{p^{2}} $. In particular, $ |T|\leq \frac{d}{p^{2}} $. By Lemma \ref{OV}(2), $ G/T $ satisfies the hypotheses of the theorem. Write $ P=T_{1}\times \cdots \times T_{m} $, where $  |T_{1}| \leq \cdots \leq |T_{m}| $, and $ T_{i} $ is a minimal normal subgroup of $ G $ for any $ i=1, 2, ..., m $. Clearly, $ m\geq 2 $ since $ |P|>d $.  As $ P\nleq \mathrm{Z}_{\UU}(G) $, we have $ p^{2}\big | |N_{m}| $.
	
	Let $ \Omega \subseteq \{1, \cdots, m \} $ be minimal such that $ d\leq |\prod\limits_{i\in \Omega}T_{i}| $. We argue that $ d= |\prod\limits_{i\in \Omega}T_{i}| $. The minimal choice of $ \Omega $ yields that $ \prod\limits_{i\in \Omega, i\not=j}T_{i} $ has order less than $ d $ for any $ j\in \Omega $. Set $ A=\prod\limits_{i\in \Omega}T_{i} $ and $ B=\prod\limits_{i\in \Omega, i\not=j}T_{i} $. Let $ S $ be a Sylow $ p $-subgroup of $ G $. Then $ A\leq S $. Let $ H $ be a normal subgroup of $ S $ of order $ d $ such that $ B<H\leq A $. Clearly, $ H^{G}=  A $, and $ H_{G}=A $ or $ B $. 
	If $ H_{G} =B $, then $ H\ntrianglelefteq G $. By hypothesis, $ H $ satisfies the partial $ \LL $-$ \Pi $-property in $ G $. Then $ |G:\mathrm{N}_{G}(H)|=|G:\mathrm{N}_{G}(HB)| $ is a $ p $-number. Since $ H\unlhd S $, we deduce that $ H\unlhd G $, a contradiction.  Hence we may assume that $ H_{G}=A $. As a consequence, $ d=|H|=|A|=|\prod\limits_{i\in \Omega}T_{i}| $, as claimed. 
	
	Note that every minimal normal subgroup of $ G $ contained in $ P $ has order at most $ \frac{d}{p^{2}} $. If $ m=2 $, then $ |P|=|T_{1}T_{2}|=d $ by the previous claim, a contradiction. Hence we may assume that $ m\geq 3 $. 
	By Lemma \ref{OV}(2), every subgroup of $ P/T_{i} $ of order $ \frac{d}{|T_{i}|} $ satisfies the partial $ \LL $-$ \Pi $-property in $ G/T_{i} $ for $ i=1, 2 $. Note that  $ P/T_{i}\nleq \mathrm{Z}_{\UU_{p}}(G/T_{i}) $ for $ i=1, 2 $. Applying the inductive hypothesis to $ G/T_{1} $ and $ G/T_{2} $, we deduce  that  $ T_{1}, T_{2}, T_{3}, ..., T_{m} $ are $ H $-isomorphic $ \mathbb{F}_{p}[H] $-modules,  $ T_{i} $ is not absolutely irreducible for any $ i=1, 2, ..., m $, $ \log_p|T_{i}T_{1}/T_{1}| $ divides $ \log_p \frac{d}{|T_{1}|} $ for any $ i=2, 3, ..., m $, and $ \log_p|T_{j}T_{2}/T_{2}| $ divides $ \log_p \frac{d}{|T_{2}|} $ for any $ j=1, 3, ..., m $.
	Therefore, $ P $ is a homogeneous $ \mathbb{F}_{p}[H] $-module.  Since $ |T_{1}|=|T_{2}|=\cdots =|T_{m}| $, we see that $ \log_p|T_{i}| $ divides $ \log_pd $ for any $ i=1, 2, ..., m $. Our proof is now complete. 
\end{proof}

\begin{theorem}\label{furthermore}
	Let $ d $ be a power of $ p $ such that $ p^{2}\leq d<|\mathrm{O}_{p',p}(G)| $. Assume that every subgroup of $ \mathrm{O}_{p',p}(G) $ of order $ d $ satisfies the partial $ \LL $-$ \Pi $-property in $ G $. Then $ G $ is $ p $-soluble. Furthermore,
	
	\begin{enumerate}[\rm(1)]
		\item $ G/\mathrm{O}_{p', p}(G) $ is $ p $-supersoluble;
		\item If in addition the $p$-rank of $G$ is greater than $1$, then $ d \geq p^{2}|P\cap \mathrm{O}_{p',\Phi}(G)| $, where $ \mathrm{O}_{p',\Phi}(G) $ is a normal subgroup of $ G $ such that $ \mathrm{O}_{p',\Phi}(G)/\mathrm{O}_{p'}(G)=\Phi(G/\mathrm{O}_{p'}(G)) $; $ \mathrm{O}_{p',p}(G)/\mathrm{O}_{p',\Phi}(G) $ is a homogeneous $ \mathbb{F}_{p}[G] $-module, and every irreducible submodule is not absolutely irreducible and has dimension dividing $ \log_p\frac{d}{|P\cap \mathrm{O}_{p',\Phi}(G)|} $.
		
	\end{enumerate}

\end{theorem}

\begin{proof}[\bf{Proof}]
    We proceed by induction on $ |G| $. It is no loss to assume that $ \mathrm{O}_{p'}(G)=1 $. Hence $ \mathrm{O}_{p',\Phi}(G)=\Phi(G)\leq \mathrm{O}_{p}(G)=\mathrm{O}_{p',p}(G) $. By Theorem \ref{p-soluble}, we know that $ G $ is $ p $-soluble.

    Assume that $ \mathrm{O}_{p}(G)\leq \mathrm{Z}_{\UU}(G) $. In view of \cite[Theorem 1.7.15]{Weinstein-1982}, we see that $ G/\mathrm{C}_{G}(\mathrm{O}_{p}(G)) $ is supersoluble.  Since $ \mathrm{O}_{p'}(G)=1 $, it follows from \cite[Theorem 6.3.2]{Gorenstein-1980} that $ \mathrm{C}_{G}(\mathrm{O}_{p}(G))\leq \mathrm{O}_{p}(G) $. 
    Thus $ G $ is  $ p $-supersoluble, and the conclusions hold. Hence  we may suppose that $ \mathrm{O}_{p}(G)\nleq \mathrm{Z}_{\UU}(G) $. In particular, $ G $ is not $ p $-supersoluble.
	
	Assume that $ \Phi(G)>1 $. Let $ T $ be a minimal normal subgroup of $ G $ contained in $ \Phi(G) $. If $ |T|\geq d $, then $ T $ has order $ d $ by Lemma \ref{Normal}. Applying Lemma \ref{ele}(2), we have $ \Phi(G)=1 $, a contradiction. If $ |T|=\frac{d}{p} $, then by Lemma \ref{leq}(1), we see that $ G/T $ is a $ p $-supersoluble. Thus $ G $ is $ p $-supersoluble, a contradiction. Therefore, we may suppose that $ |T|\leq \frac{d}{p^{2}} $, i.e., $ \frac{d}{|T|}\geq p^{2} $. By Lemma \ref{OV}(2), every subgroup of $ P/T $ of order $ \frac{d}{|T|} $ satisfies the partial $ \LL $-$ \Pi $-property in $ G/T $. Hence the results follow by applying the inductive hypothesis to $ G/T $.
	
	Assume that $ \Phi(G) = 1 $. Then $ \mathrm{O}_{p}(G)=T_{1}\times T_{2}\times \cdots \times T_{s} $, where $ T_{i} $ is a minimal normal subgroup of $ G $ for all $ 1\leq i\leq s $. Since $ p^{2}\leq d<\mathrm{O}_{p}(G) $, it follows from Theorem \ref{dividing} that (2) holds. Then $ T_{1}\cong T_{2}\cong \cdots \cong T_{s} $ and $ \log_p|T_{i}|$ divides $ \log_pd $ for all $ 1\leq i\leq s $. We only need to show that $ G/\mathrm{O}_{p}(G) $ is $ p $-supersoluble. If $ \frac{d}{p}=|T_{1}| $, then $ (\log_pd -1)\big|\log_pd $. Thus $ d=p^{2} $ and $ p=|T_{1}|=|T_{2}|=\cdots =|T_{s}| $. As a consequence, $ \mathrm{O}_{p}(G)\leq \mathrm{Z}_{\UU}(G) $, a contradiction. Hence we may assume that $ \frac{d}{p^{2}}\geq |T_{1}| $. Observe that $ p^{2}\leq \frac{d}{|T_{1}|}< |\mathrm{O}_{p}(G)/T_{1}|= |\mathrm{O}_{p}(G/T_{1})| $.  By Lemma \ref{OV}(2), the hypotheses are inherited by $ G/T_{1} $. By induction, we can get that $ G/T_{1}\big /\mathrm{O}_{p}(G/T_{1}) $ is $ p $-supersoluble. Therefore $ G/\mathrm{O}_{p}(G) $ is  $ p $-supersoluble, as desired.
\end{proof}

\begin{proof}[\bf{Proof of Theorem \ref{two}}]
	Applying Theorem \ref{super}, we may assume that $ p^{2}\leq d \leq \sqrt{|P|} $. In view of Theorem \ref{p-soluble}, $ G $ is $ p $-soluble. By Lemma \ref{less}, we have that  $ |G/\mathrm{O}_{p',p}(G)|_{p} < |\mathrm{O}_{p',p}(G)|_{p} $. It follows that $ d\leq \sqrt{|P|}< |\mathrm{O}_{p',p}(G)|_{p} $. By Lemma \ref{furthermore}, the conclusions hold.
\end{proof}

\begin{proof}[\bf{Proof of Theorem \ref{three}}]
	 Assume that $ G $ is not $ p $-supersoluble. By Theorem \ref{p-soluble}, $ G $ is $ p $-soluble. Clearly, every minimal normal subgroup of $ G $ is a $ p $-group since $ \mathrm{O}_{p'}(G)=1 $. If $ |P|=p^{2} $, then $ P $ is a minimal  normal subgroup of $ G $. Hence $ G $ is of type (2). So we may suppose that $ |P|\geq p^{3} $.
	

	Let $ T $ be a minimal normal subgroup of $ G $. Then $ T\leq P $. 
	In view of  Lemma \ref{Normal}, we can get that $ |T|\leq p^{2} $. If $ |T|=p $, then $ G $ is $ p $-supersoluble by Lemma \ref{leq}(1), a contradiction. Therefore $ |T|=p^{2} $.  By Lemma \ref{ele}(2), it follows that $ \Phi(G)=1 $ and $ \mathrm{Soc}(G) =\mathrm{O}_{p}(G) $ is a homogeneous $ \mathbb{F}_{p}[G] $-module. As a  consequence, $ \mathrm{C}_{G}(\mathrm{O}_{p}(G)) = \mathrm{C}_{G}(T) $. Note that $ \mathrm{O}_{p'}(G)=1 $ and $ G $ is $ p $-soluble. By \cite[Theorem 6.3.2]{Gorenstein-1980}, we have $ \mathrm{O}_{p}(G)=\mathrm{C}_{G}(\mathrm{O}_{p}(G))=\mathrm{C}_{G}(T) $. Therefore, $ G/\mathrm{O}_{p}(G)\lesssim \mathrm{GL}(2, p) $.

	Suppose that $ |\mathrm{O}_{p}(G)|< |P| $. Write $ \overline{G}=G/\mathrm{O}_{p}(G) $. Observe that $ \overline{G} $ is a $ p $-soluble subgroup of $ \mathrm{GL}(2, p) $ with $ \mathrm{O}_{p}(\overline{G}) = 1 $ and $ p\big| |\overline{G}| $. Investigating the subgroups of $ \mathrm{GL}(2, p) $, we deduce that
	either $ p = 2 $ and $ \overline{G}\cong \mathrm{GL}(2, 2) \cong \mathrm{Sym}(3) $, or $ p = 3 $ and $ \mathrm{SL}(2, 3) \lesssim \overline{G} \lesssim \mathrm{GL}(2, 3) $. In the both cases, $ T $ is an absolutely irreducible $ \mathbb{F}_{p}[G] $-module. By Lemma \ref{ele}(3), we have $ T=\mathrm{O}_{p}(G) $. Hence  $ G $ is of type (3) or type (4).
	
	Suppose that $ |\mathrm{O}_{p}(G)| = |P| $. Then $ P $ is a normal Sylow $ p $-subgroup of $ G $ with $ |P|\geq p^{3} $. Since $ G $ is not $ p $-supersoluble, it follows from  Theorem \ref{furthermore}(2) that $ P $ is a homogeneous $ \mathbb{F}_{p}[G] $-module with all its irreducible $ \mathbb{F}_{p}[G] $-submodules being not absolutely irreducible and having dimension $ 2 $. Applying Lemma \ref{cyclic}, we conclude that $ G/P $ is  cyclic. Therefore, $ G $ is of type (5). 
\end{proof}



\section*{Acknowledgments}

    The first author thanks the China Scholarship Council and the Departament de Matem$ \grave{\mathrm{a}} $tiques of the Universitat de Val$ \grave{\mathrm{e}} $ncia for its hospitality.


    \small

\begin{thebibliography}{10}
	
	\bibitem{Adolfo-2011}
	A.~Ballester-Bolinches, R.~Esteban-Romero, and Y.~Li.
	\newblock On second minimal subgroups of {Sylow} subgroups of finite groups.
	\newblock {\em J. Algebra}, 342(1):134--146, 2011.
	
	\bibitem{Adolfo-2015}
	A.~Ballester-Bolinches, L.~M. Ezquerro, Y.~Li, and N.~Su.
	\newblock On partial {CAP}-subgroups of finite groups.
	\newblock {\em J. Algebra}, 431:196--208, 2015.
	
	\bibitem{ball-2009}
	A.~Ballester-Bolinches, L.~M. Ezquerro, and A.~N. Skiba.
	\newblock Local embeddings of some families of subgroups of finite groups.
	\newblock {\em Acta Math. Sin. (Engl. Ser.)}, 25(6):869--882, 2009.
	
	\bibitem{Adolfo-JPA}
	A.~Ballester-Bolinches, L.~M. Ezquerro, and A.~N. Skiba.
	\newblock On second maximal subgroups of {S}ylow subgroups of finite groups.
	\newblock {\em J. Pure Appl. Algebra}, 215(4):705--714, 2011.
	
	\bibitem{Chen-2013}
	X.~Chen and W.~Guo.
	\newblock On the partial {{\(\Pi\)}}-property of subgroups of finite groups.
	\newblock {\em J. Group Theory}, 16(5):745--766, 2013.
	
	\bibitem{Chen-xiaoyu-2016}
	X.~Chen and W.~Guo.
	\newblock On {$\Pi$}-supplemented subgroups of a finite group.
	\newblock {\em Comm. Algebra}, 44(2):731--745, 2016.
	
	\bibitem{MR1169099}
	K.~Doerk and T.~Hawkes.
	\newblock {\em Finite Soluble Groups}, volume~4 of {\em De Gruyter Expositions
		in Mathematics}.
	\newblock Walter de Gruyter \& Co., Berlin, 1992.
	
	\bibitem{Weinstein-1982}
	M.~Weinstein (Ed.).
	\newblock {\em Between Nilpotent and Solvable}.
	\newblock Polygonal Publ. House, Washington, N. J., 1982.
	
	\bibitem{Gorenstein-1980}
	D.~Gorenstein.
	\newblock {\em Finite Groups}.
	\newblock Chelsea Publishing Co., New York, second edition, 1980.
	
	\bibitem{Guo-2000}
	W.~Guo.
	\newblock {\em The Theory of Classes of Groups}, volume 505 of {\em Mathematics
		and its Applications}.
	\newblock Kluwer Academic Publishers Group, Dordrecht; Science Press Beijing,
	Beijing, 2000.
	
	\bibitem{Guo2015}
	W.~Guo.
	\newblock {\em Structure Theory for Canonical Classes of Finite Groups}.
	\newblock Springer, Heidelberg, 2015.
	
	\bibitem{Guo-Xie-Li}
	W.~Guo, F.~Xie, and B.~Li.
	\newblock Some open questions in the theory of generalized permutable
	subgroups.
	\newblock {\em Sci. China Ser. A}, 52(10):2132--2144, 2009.
	
	\bibitem{Qian2012}
	G.~Qian.
	\newblock A character theoretic criterion for a {$p$}-closed group.
	\newblock {\em Israel J. Math.}, 190:401--412, 2012.
	
	\bibitem{Qian-2021}
	G.~Qian.
	\newblock Local partial covering subgroups in finite groups.
	\newblock {\em J. Algebra}, 572:129--145, 2021.
	
	\bibitem{Qian-2020}
	G.~Qian and Y.~Zeng.
	\newblock On partial cap-subgroups of finite groups.
	\newblock {\em J. Algebra}, 546:553--565, 2020.
	
	\bibitem{Qiu}
	Z.~Qiu, J.~Liu, and G.~Chen.
	\newblock On the partial {$\Pi$}-property of some subgroups of prime power
	order of finite groups.
	\newblock {\em Mediterr. J. Math.}, 21(2):Paper No. 61, 2024.
	
	\bibitem{Su-2014}
	N.~Su, Y.~Li, and Y.~Wang.
	\newblock A criterion of {$p$}-hypercyclically embedded subgroups of finite
	groups.
	\newblock {\em J. Algebra}, 400:82--93, 2014.
	
	\bibitem{Ward}
	H.~N. Ward.
	\newblock Automorphisms of quaternion-free 2-groups.
	\newblock {\em Math. Z.}, 112:52--58, 1969.
	
	\bibitem{Wolf-1984}
	T.~R. Wolf.
	\newblock Sylow-{$p$}-subgroups of {$p$}-solvable subgroups of {${GL}(n,p)$}.
	\newblock {\em Arch. Math.}, 43:1--10, 1984.
	
\end{thebibliography}

\end{sloppypar}
\end{document}